\documentclass[11pt,titlepage]{article}
\usepackage{amsmath, amssymb}
\usepackage{mathtools}
\usepackage{algorithm}
\usepackage{algpseudocode}
\usepackage{fmtcount}
\usepackage{graphicx}
\usepackage{enumitem} 
\usepackage{subcaption}
\usepackage{afterpage}
\usepackage[noadjust]{cite}
\usepackage[font=small,labelfont=bf]{caption}
\usepackage[labelfont=bf]{caption}
\usepackage{setspace}
\allowdisplaybreaks
\usepackage[title,titletoc]{appendix}
\usepackage{epsfig}
\usepackage{graphicx}
\usepackage{hhline}
\usepackage{latexsym}
\usepackage{amssymb}
\usepackage{amsmath,amscd}
\usepackage{multirow} 
\usepackage{array}
\usepackage{caption}
\usepackage{subcaption}
\usepackage{color}
\usepackage{natbib}
\usepackage{rotating}
\usepackage{graphicx,lscape}
\usepackage{amsmath, amsthm, amssymb}
\usepackage{graphicx}
\usepackage{epsfig}
\usepackage{graphics,color, graphicx}
\usepackage{latexsym}
\usepackage{fullpage}
\usepackage{booktabs,fixltx2e}
\usepackage[flushleft]{threeparttable}
\usepackage{amssymb,amsmath,amscd}
\usepackage{fancyvrb}
\usepackage{pdfpages}
\usepackage{fmtcount}
\usepackage{url}
\usepackage[margin=1in]{geometry}

\newtheorem{proposition}{Proposition}[section]
\newtheorem{assumption}{Assumption}[section]

\newtheorem{corollary}{Corollary}[section]

\theoremstyle{definition}

\theoremstyle{definition}
\newtheorem{remark}{Remark}[section]
\usepackage{setspace}
\title{\bf 
\bf Geometric quantile-based measures of multivariate distributional characteristics 
	\medskip
}

\author{
	\sc  Ha-Young Shin and Hee-Seok Oh\\ 
	Department of Statistics\\
	Seoul National University\\
	Seoul 08826, Korea \\
	\\
}

\begin{document}
	\maketitle
	
	\begin{abstract}
	Several new geometric quantile-based measures for multivariate dispersion, skewness, kurtosis, and spherical asymmetry are defined. These measures differ from existing measures, which use volumes, and are easy to calculate. Some theoretical justification is given, followed by experiments illustrating that they are sensible measures of these distributional characteristics and some basic empirical justification for bootstrapped confidence regions.
		\vspace{\baselineskip}
		
		
	\end{abstract}
	
	\pagenumbering{arabic}

\section{Introduction}

	Denote the open ball in $\mathbb{R}^n$ with center $p$ and radius $r$ by $B_r^n(p)$. \cite{Chaudhuri1996} defined a geometric notion of quantiles for multidimensional data by defining the $u$-quantile, $u\in B_1^n(0)$, to be the value of $p$ that minimizes $E[\lVert X-p\rVert+\langle u,X-p\rangle]$, which generalizes the univariate quantile loss function. However, a common criticism of these geometric quantiles is that the isoquantile contours need not follow the shape of the distribution very well when it is not spherically symmetric, especially for extreme quantiles, which also need not be contained in the convex hull of the distribution, as shown by \cite{Girard2017}. To mitigate this problem, \cite{Chakraborty2003} devised a procedure to transform the data to have roughly isotropic covariance.

A potential use of geometric quantiles that does not need this procedure is in defining robust measures of multivariate centrality, dispersion, skewness, and kurtosis. For real-valued data, the first four moments are used to define measures of centrality, dispersion, skewness, and kurtosis, respectively. Meanwhile, some quantile-based measures of these quantities also exist in the literature. Denoting the $\tau$-quantile for univariate data by $Q_{\tau}^X$, $\tau\in(0,1)$, the median $Q_{0.5}^X$ is a centrality measure, and 
\begin{equation*} 
\delta_0^X(\beta):=Q_{(1+\beta)/2}^X-Q_{(1-\beta)/2}^X
\end{equation*}
for $\beta\in(0,1)$ measures dispersion; this is the interquartile range when $\beta=1/2$. A standard quantile-based measure of skewness is 
\begin{align*}
\gamma_0^X(\beta):=\frac{Q_{(1+\beta)/2}^X+Q_{(1-\beta)/2}^X-2Q_{1/2}^X}{\delta_0^X(\beta)},
\end{align*}
where $\beta$ is typically set to $1/2$ \citep{Groeneveld1984}, and kurtosis (or tailedness) can be measured by 
\begin{align*}
\kappa_0^X(\beta,\beta'):=\frac{\delta_0^X(\beta')}{\delta_0^X(\beta)}
\end{align*}
for $0<\beta<\beta'<1$ \citep{Balanda1988}.

For multivariate data, \cite{Chaudhuri1996} defined measures of multivariate dispersion, skewness, and kurtosis using the volume of a region enclosed by an isoquantile contour, by which we mean the set of all $u$-quantiles for fixed $\lVert u\rVert\in(0,1)$. However, these measures are only briefly mentioned in that paper and are computationally complex for two reasons. First, since geometric quantiles cannot, in general, be found analytically, it is not feasible to have an exact description of an entire contour as that would require numerically calculating the $u$-quantile for every $u\in B_1^n(0)$ of a given size. Second, computing the volume enclosed by an arbitrary hypersurface is non-trivial even in just two dimensions, especially considering that the enclosed region may not be convex.

In Section \ref{quantile}, we define two new measures for each of these characteristics, one based on averages and the other on suprema, which bypass the second issue altogether by avoiding volumes, while mitigating the first. Furthermore, we introduce a measure of spherical asymmetry and investigate some theoretical underpinnings of these measures. Section \ref{experiments} demonstrates the use of these measures, providing some visualization, and calculates bootstrapped confidence regions and their coverage. The code for this paper, implemented in Python, can be found at \url{https://github.com/hayoungshin1/Quantile-based-measures}.

\section{Quantile-based measures of multivariate dispersion, skewness, kurtosis, and spherical asymmetry} \label{quantile}

For an $n$-dimensional random vector $X$ with unique $u$-quantiles for all $u\in B_1^n(0)$, define a map $q^X:B_1^n(0)\rightarrow \mathbb{R}^n$ that sends $u$ to the $u$-quantile of $X$, and denote $q^X(0)$ by $m^X$. For such an $X$, we introduce two measures of dispersion as follows
\begin{align*}
\delta_1^X(\beta)&:=\sup_{\xi\in S^{n-1}}{\lVert q^X(\beta \xi)-q^X(-\beta \xi)\rVert},\\
\delta_2^X(\beta)&:=\frac{1}{SA(n-1)}\int_{S^{n-1}}\lVert q^X(\beta \xi)-q^X(-\beta \xi)\rVert d\xi,~~\beta\in (0,1), 
\end{align*}
where $SA(n-1)$ is the surface area of the unit $(n-1)$-sphere. For skewness, we suggest the following two measures 
\begin{align*}
\gamma_1^X(\beta)&:=\frac{\sup_{\xi\in S^{n-1}}\lVert q^X(\beta \xi)+q^X(-\beta \xi)-2m^X\rVert}{\delta_1^X(\beta)},\\
\gamma_2^X(\beta)&:=\frac{\int_{S^{n-1}} (q^X(\beta \xi)-m^X) d\xi/SA(n-1)}{\delta_2^X(\beta)},~~\beta\in (0,1),
\end{align*}
and for kurtosis, 
\[
\kappa_1^X(\beta,\beta'):=\frac{\delta_1^X(\beta')}{\delta_1^X(\beta)}~~~\mbox{and}~~~
\kappa_2^X(\beta,\beta'):=\frac{\delta_2^X(\beta')}{\delta_2^X(\beta)},
\]
for $0<\beta<\beta'<1$. 
When $n=1$, it is clear that $\delta_0^X(\beta)=\delta_1^X(\beta)=\delta_2^X(\beta)$, $\lvert\gamma_0^X(\beta)\rvert=\gamma_1^X(\beta)=\lvert\gamma_2^X(\beta)\rvert$, and $\kappa_0^X(\beta,\beta')=\kappa_1^X(\beta,\beta')=\kappa_2^X(\beta,\beta')$. Unlike $\delta_0^X(\beta)$ and $\kappa_0^X(\beta,\beta')$, $\gamma_0^X(\beta)$ can be positive or negative, so its generalization to higher dimensions should be a vector as $\gamma_2^X(\beta)$ is. However, we have also included $\gamma_1^X(\beta)$, which is a scalar and thus, strictly speaking, generalizes $\lvert \gamma_0^X(\beta)\rvert$ rather than $\gamma_0^X(\beta)$ itself.

Under the following assumption, Fact 2.1.1 of \cite{Chaudhuri1996} guarantees that a unique $u$-quantile of $X$ exist for all $u\in B_1^n(0)$. 

\begin{assumption}\label{support}
    The support of the $n$-dimensional random vector $X$, $n\geq2$, is not contained on a single line.
\end{assumption}

Proposition 6.1 in \cite{Konen2022} shows that for such an $X$, $q^X$ is continuous. In the rest of this section, we assume that $X$ satisfies Assumption \ref{support}.

We now consider the spherical asymmetry of a distribution, which is a relevant property for multivariate distributions. We propose the spherical asymmetry of a distribution as 
\begin{equation*}
    \alpha^X(\beta):=\log\bigg(\frac{\sup_{\xi\in S^{n-1}}\lVert q^X(\beta \xi)-m^X\rVert}{\inf_{\xi\in S^{n-1}}\lVert q^X(\beta \xi)-m^X\rVert}\bigg).
\end{equation*}
By transforming the data as in \cite{Chakraborty2001} to have roughly isotropic covariance, $\alpha^X(\beta)$ can also be used to measure elliptical asymmetry. All of our measures can be used to test properties of distributions, such as whether a distribution is skewed or spherically symmetric or whether one distribution is more spread out than another. 

Note that all of the integrands used to define these measures are continuous thanks to Assumption \ref{support} and thus have finite integrals because the domain of integration is the compact $S^{n-1}$. Next, consider the following two assumptions.

\begin{assumption} \label{nonzero1}
There exists some $\xi\in S^{n-1}$ for which $q^X(\beta\xi)\neq q^X(-\beta\xi)$.
\end{assumption}
\begin{assumption} \label{nonzero2}
For all $\xi\in S^{n-1}$, $q^X(\beta\xi)\neq m^X$.
\end{assumption}

Clearly $\gamma_1^X(\beta)$ and $\kappa_1^X(\beta,\beta')$ are defined in $\mathbb{R}$ if and only if Assumption \ref{nonzero1} holds. In fact, $\gamma_2^X(\beta)$ and $\kappa_2^X(\beta,\beta')$ are defined in $\mathbb{R}$ under the same condition and only under this condition since the continuity of $q^X$ implies that the map $\xi\mapsto\lVert q^X(\beta \xi)-q^X(-\beta \xi)\rVert$ is positive on a non-null set (according to the standard volume measure on $S^{n-1}$). On the other hand, $\alpha(\beta)$ is defined in $\mathbb{R}$ if and only if Assumption \ref{nonzero2} holds. Both of these assumptions can be guaranteed when $X$ has a non-atomic distribution not supported on a single line, in which case $q^X$ is a homeomorphism and hence injective; see Theorem 6.2 in \cite{Konen2022}.

$X$ is said to have a non-skewed distribution if $X-m^X$ and $-(X-m^X)$ have identical distributions and a spherically symmetric distribution if $X-m^X$ and $A(X-m^X)$ have identical distributions for all orthogonal $n\times n$ matrices $A$. Then, we can show that the skewness and spherical asymmetry measures behave as desired when the distribution of $X$ is non-skewed or spherically symmetric, respectively. We can also show that our measures have desirable invariance and equivariance properties with respect to scaling, translation, and orthogonal transformations like rotation and reflection.

\begin{proposition}\label{prop}
Let $\beta\in (0,1)$ and Assumption \ref{support} hold for $X=Y$.
\begin{itemize}
    \item[(a)] If $Y$ has a non-skewed distribution, then under Assumption \ref{nonzero1} for $X=Y$, $\gamma_1^Y(\beta)=\gamma_2^Y(\beta)=0$.

    \item[(b)] If $Y$ has a spherically symmetric distribution, then under Assumption \ref{nonzero2} for $X=Y$, $\alpha^Y(\beta)=0$.

    \item[(c)] Let $A$ be an orthogonal $n\times n$ matrix, $c>0$, and $v\in\mathbb{R}^n$. Then, $\delta_1^{cAY+v}(\beta)=c\delta_1^{Y}(\beta)$ and $\delta_2^{cAY+v}(\beta)=c\delta_2^{Y}(\beta)$. Under Assumption \ref{nonzero1} for $X=Y$,  $\gamma_1^{cAY+v}(\beta)=\gamma_1^{Y}(\beta)$, $\gamma_2^{cAY+v}(\beta)=\gamma_2^{Y}(\beta)$, $\kappa_1^{cAY+v}(\beta,\beta')=\gamma_1^{Y}(\beta)$, and $\kappa_2^{cAY+v}(\beta,\beta')=\gamma_2^{Y}(\beta)$. Under Assumption \ref{nonzero2} for $X=Y$, $\alpha^{cAY+v}(\beta)=\alpha^{Y}(\beta)$
    \end{itemize}
\end{proposition}
\begin{proof}
(a) Fact 2.2.1 of \cite{Chaudhuri1996} details the equivariance of geometric quantiles to translation and orthogonal transformations. The result follows immediately from
\begin{align*}
    q^Y(-\beta\xi)-m^Y=q^{Y-m^Y}(-\beta\xi)=q^{-(Y-m^Y)}(-\beta\xi)=-q^{-(Y-m^Y)}(\beta\xi)&=-q^{Y-m^Y}(\beta\xi) \\
    &=-(q^Y(\beta\xi)-m^Y),
\end{align*}
where the five equalities follow from translation equivariance, non-skewness, orthogonal equivariance, non-skewness, and translation equivariance, respectively.

(b) For any $\xi,\xi'\in S^{n-1}$, there exists some orthogonal $A$ for which $\xi'=A\xi$. Then, 
\begin{align*}
    q^Y(\beta \xi')-m^Y=q^{Y-m^Y}(\beta\xi')=q^{A(Y-m^Y)}(\beta\xi')=Aq^{A(Y-m^Y)}(\beta\xi)&=Aq^{Y-m^Y}(\beta\xi) \\
    &=A(q^Y(\beta\xi)-m^Y),
\end{align*}
where the five equalities follow from translation equivariance, spherical symmetry, orthogonal equivariance, spherical symmetry, and translation equivariance, respectively. Thus, the orthogonality of $A$ implies that $\lVert q^Y(\beta \xi')-m^Y\rVert=\lVert q^Y(\beta \xi)-m^Y\rVert$ for all $\xi,~\xi'\in S^{n-1}$, from which the result follows.

(c) Note that Assumption \ref{support} holds for $X=cAY+v$. By the aforementioned equivariance of geometric quantiles with respect to translation and orthogonal transformations, in addition to equivariance with respect to scaling, 
\begin{equation}\label{equi}
q^{cAY+v}(\beta\xi)=cq^Y(\beta A\xi)+v;
\end{equation}
see Facts 2.2.1 and 2.2.2 of \cite{Chaudhuri1996}. This implies that Assumption \ref{nonzero1} holds for $X=cAY+v$ whenever it does so for $X=Y$, and similarly for Assumption \ref{nonzero2}. Given $\xi\in S^{n-1}$, $\xi'$ is in $S^{n-1}$ if and only if $\xi'=A\xi$ for some orthogonal $A$. The results are then easily derived from this fact and (\ref{equi}).
\end{proof}

For each positive integer $k$, define a set $\Xi_k\subset S^{n-1}$ of cardinality $k$. Then, our measures can be approximated by 
\begin{equation}\begin{gathered}\label{approximations}
\hat{\delta}_1^X(\beta;\Xi_k):=\max_{\xi\in \Xi_k}{\lVert q^X(\beta \xi)-q^X(-\beta \xi)\rVert},~~
\hat{\delta}_2^X(\beta;\Xi_k):=\frac{1}{k}\sum_{\xi\in \Xi_k}{\lVert q^X(\beta \xi)-q^X(-\beta \xi)\rVert}, \\
\hat{\gamma}_1^X(\beta;\Xi_k):=\frac{\max_{\xi\in \Xi_k}{\lVert q^X(\beta \xi)+q^X(-\beta \xi)-2m^X\rVert}}{\hat{\delta}_1^X(\beta;\Xi_k)},\\
\hat{\gamma}_2^X(\beta;\Xi_k):=\frac{(1/k)\sum_{\xi\in \Xi_k}{(q^X(\beta \xi)-m^X)}}{\hat{\delta}_2^X(\beta;\Xi_k)}, \\
\hat{\kappa}_1^X(\beta,\beta';\Xi_k):=\frac{\hat{\delta}_1^X(\beta';\Xi_k)}{\hat{\delta}_1^X(\beta;\Xi_k)}, ~~
\hat{\kappa}_2^X(\beta,\beta';\Xi_k):=\frac{\hat{\delta}_2^X(\beta';\Xi_k)}{\hat{\delta}_2^X(\beta;\Xi_k)}, \\
\hat{\alpha}^X(\beta;\Xi_k):=\log\bigg(\frac{\max_{\xi\in \Xi_k}\lVert q^X(\beta \xi)-m^X\rVert}{\min_{\xi\in \Xi_k}\lVert q^X(\beta \xi)-m^X\rVert}\bigg).
\end{gathered}\end{equation}
For each $k$, define $\mu_k$ to be the empirical measure corresponding to $\Xi_k$; that is, the probability measure that assigned a mass of $1/k$ to each point in $\Xi_k$. Define $\mu$ to the uniform probability measure on $S^{n-1}$, or equivalently the normed volume measure of $S^{n-1}$; that is, for any measurable $F\subset S^{n-1}$, $\mu(F)$ is the $(n-1)$-dimensional volume of $F$ according to the standard volume measure on $S^{n-1}$ divided by $SA(n-1)$. The sequence of sets $\Xi_1,\Xi_2,\ldots$ is called equidistributed if the corresponding sequence of empirical measures $\mu_1,\mu_2,\ldots$ converge weakly to $\mu$.

\begin{proposition} \label{approx}
    Let $\Xi_1,\Xi_2,\ldots\subset S^{n-1}$ be an equidistant sequence, $\beta\in(0,1)$, and Assumption \ref{support} hold for $X=Y$. Then, $\lim_{k\rightarrow\infty}\hat{\delta}_1^Y(\beta;\Xi_k)=\delta_1^Y(\beta)$ and $\lim_{k\rightarrow\infty}\hat{\delta}_2^Y(\beta;\Xi_k)=\delta_2^Y(\beta)$. Under Assumption \ref{nonzero1} for $X=Y$, $\lim_{k\rightarrow\infty}\hat{\gamma}_1^Y(\beta;\Xi_k)=\gamma_1^Y(\beta)$, $\lim_{k\rightarrow\infty}\hat{\gamma}_2^Y(\beta;\Xi_k)=\gamma_2^Y(\beta)$, $\lim_{k\rightarrow\infty}\hat{\kappa}_1^Y(\beta,\beta';\Xi_k)=\kappa_1^Y(\beta,\beta')$, and $\lim_{k\rightarrow\infty}\hat{\kappa}_2^Y(\beta,\beta';\Xi_k)=\kappa_2^Y(\beta,\beta')$. Under Assumption \ref{nonzero2} for $X=Y$, $\lim_{k\rightarrow\infty}\hat{\alpha}^Y(\beta;\Xi_k)=\alpha^Y(\beta)$.
\end{proposition}
\begin{proof}
    For a continuous function $f:S^{n-1}\rightarrow\mathbb{R}$, there exists some $\xi'$ in the compact set $S^{n-1}$ for which $\sup_{\xi\in S^{n-1}}f(\xi)=f(\xi')$. By continuity, for any $\epsilon>0$, there is some open $\eta>0$ for which $v\in B_\eta^n(\xi')\cap S^{n-1}$ implies $f(v)>f(\xi')-\epsilon$. Then, by the Portmanteau theorem, $\lvert\Xi_k\cap B_\eta^n(\xi')\rvert=\mu_n(B_\eta^n(\xi')\cap S^{n-1})\rightarrow \mu(B_\eta^n(\xi')\cap S^{n-1})>0$ as $k\rightarrow\infty$, implying that $\Xi_k\cap B_\eta^n(\xi')$ is non-empty for sufficiently large $k$. Therefore, for sufficiently large $k$, $\max_{\xi\in\Xi_k}f(\xi)>\sup_{\xi\in S^{n-1}}f(\xi)-\epsilon$ while $\max_{\xi\in\Xi_k}f(\xi)\leq \sup_{\xi\in S^{n-1}}f(\xi)$ since $\Xi_k\subset S^{n-1}$. This can be done for any $\epsilon$, so $\lim_{k\rightarrow\infty}\max_{\xi\in\Xi_k}f(\xi)=\sup_{\xi\in S^{n-1}}f(\xi)$. The analogous result replacing $\sup$ with $\inf$ and $\max$ with $\min$ can be shown similarly, and because $q^Y$ is continuous, $\hat{\delta}_1^Y(\beta;\Xi_k)$, $\hat{\gamma}_1^Y(\beta;\Xi_k)$, $\hat{\kappa}_1^Y(\beta,\beta';\Xi_k)$, and $\hat{\alpha}_1^Y(\beta;\Xi_k)$ converge to the appropriate terms as $k\rightarrow\infty$.

    The aforementioned continuous $f$ is also bounded because its domain $S^{n-1}$ is compact, and so the Portmanteau theorem also implies that $\sum_{\xi\in\Xi_k}f(\xi)=\int f d\mu_k\rightarrow \int f d\mu=\int_{S^{n-1}}f(\xi)d\xi$. Then, the rest of the statement immediately follows from the continuity of $q^Y$ and each of its $n$ component functions.
\end{proof}

\begin{corollary}\label{as}
    The conclusion of Proposition \ref{approx} holds on a set of probability 1 if for each positive integer $k$, $\Xi_k=\{\xi_1,\ldots,\xi_k\}$, where $\xi_1,\xi_2,\ldots$ are independent and identically distributed random elements drawn from the uniform distribution on $S^{n-1}$.
\end{corollary}
\begin{proof}
    $\Xi_1,\Xi_2,\ldots$ is almost surely equidistributed because $S^{n-1}$ is separable; see for example Theorem 11.4.1 of \cite{Dudley2002} and Theorem 3 of \cite{Varadarajan1958}.
\end{proof}

\begin{remark}\label{weak}
When $n=2$, letting $\Xi_k=\{(\cos(2\pi l/k),\sin(2\pi l/k)):l=1,\ldots,k\}$ clearly results in an equidistributed sequence. The situation is more complicated when $n\geq 3$, but with the above corollary, we can guarantee the desired convergences with probability 1. One can easily draw an element from the uniform distribution on $S^{n-1}$ by generating a vector from an $n$-variate normal distribution with non-zero, isotropic variance and dividing the vector by its norm.
\end{remark}

\begin{remark}\label{comparison}
    Our measures have certain advantages over existing measures of these multivariate characteristics. As mentioned in the introduction, a major disadvantage of those of \cite{Chaudhuri1996} is the difficulty of computation. Most other existing measures, such as those of \cite{Mardia1970}, which we make use of in Section \ref{sexperiments} and are the best-known measures of multivariate skewness, and \cite{Mori1994}, are calculated with moments and subject to the standard concerns about robustness and sensitivity to outliers. On the other hand, as \cite{Konen2024} show, spatial quantiles are robust, and thus our measures are also expected to have desirable robustness properties.
\end{remark}

\section{Numerical experiments} \label{experiments}

\subsection{Performance of the proposed measures} \label{sexperiments}

In this section, we want to explore the performance of each of the seven measures. To do this, we generated datasets from 12 different distributions, three for each of the four characteristics, in the Euclidean plane. For a given characteristic, each distribution has an associated level $\nu\in\{0,1,2\}$ and is chosen in such a way that a plausible measure of that characteristic should decrease in magnitude as $\nu$ increases. Denoting a vector drawn from a two-dimensional distribution as $(v^1,v^2)^T$, these distributions are as follows. In all cases, $v^1$ and $v^2$ are independent. For dispersion, $v^1\sim\mathcal{N}(0,1)/2$ for all $\nu$, and $v^2\sim\mathcal{N}(0,1)$, $\mathcal{N}(0,1)/2$, $\mathcal{N}(0,1)/4$ for $\nu=0,1,2$, respectively. Denoting by $\mathcal{SN}(\rho)$ the skew-normal distribution with location parameter $0$, scale parameter $1$, and shape (skewness) parameter $\rho$, for skewness, $v^1\sim \mathcal{SN}(8)$, $\mathcal{SN}(2)$, $\mathcal{SN}(0)$, and $v^2\sim \mathcal{SN}(8)/2$, $\mathcal{SN}(2)/2$, $\mathcal{SN}(0)/2$ for $\nu=0,1,2$, respectively. Denoting by $t(l)$ the $t$-distribution with $l$ degrees of freedom, for kurtosis, $v^1\sim t(5)$, $t(7)$, $t(13)$, and $v^2\sim t(5)/2$, $t(7)/2$, $t(13)/2$ for $\nu=0,1,2$, respectively. For spherical asymmetry, $v^1,~v^2\sim \mathcal{SN}(8)$, $\mathcal{SN}(2)$, $\mathcal{SN}(0)$ for $\nu=0,1,2$, respectively.

Because skew-normal distributions are considered to increase in skewness with the shape parameter and $t$-distributions are considered to increase in kurtosis with the degrees of freedom, the rationales for choosing distributions are clear. Because of existing concerns about geometric quantiles for non-isotropic distributions, for the dispersion datasets, we decided to decrease dispersion only in $v^2$ to show that $\delta_1^X(\beta)$ and $\delta_2^X(\beta)$ work well even when the data do not have isotropic covariance. We divided the distributions for $v^2$ by 2 for the skewness and kurtosis datasets for the same reason. In contrast, for the spherical asymmetry datasets, we have not done so to show that $\alpha^X(\beta)$ can detect changes in spherical asymmetry even when the covariance remains roughly isotropic.

From each distribution, we generated samples of 300 vectors $v_1,\ldots,v_N$ ($N=300$) and estimated the population values of various measures below. Over 400 simulations, we computed the average values of the estimates and their standard errors. 

First, we did this with sample versions of the Fr\'echet variance and the multivariate skewness and kurtosis of \cite{Mardia1970} for the relevant distributions,
\begin{align} \label{samples}
    \frac{1}{N}\sum_{i=1}^N\lVert v_i-\bar v\rVert^2,~~ \frac{1}{N^2}\sum_{i=1}^N\sum_{j=1}^N[(v_i-\bar v)^TS^{-1}(v_j-\bar v)]^3,~~ \frac{1}{N}\sum_{i=1}^N[(v_i-\bar v)^TS^{-1}(v_i-\bar v)]^2,
\end{align}
respectively, where $\bar v$ is the sample mean of $\{v_1,\ldots,v_N\}$ and $S$ is the sample covariance matrix. The results are listed in Table \ref{t1}, which shows that, according to existing standard measures, these three characteristics tend to decrease with $\nu$, verifying our choice of distributions as a sensible one for our purposes.

\begin{table}
    \centering
    \caption{Estimates for (\ref{samples}) for the relevant distributions, where $(\beta,\beta')=(0.2,0.8)$ for the kurtosis measures and $\beta=0.5$ otherwise, with standard errors in parentheses.}
    {\small
\begin{tabular}{ |c||c|c|c| } 
\hline
$\nu$ & Fr\'echet variance & Mardia's multivariate skewness & Mardia's multivariate kurtosis \\
\hline
0 & 1.24887 (0.00426) & 1.78491 (0.02267) & 14.44701 (0.29255) \\ 
1 & 0.50001 (0.00150) & 0.48598 (0.01005) & 11.38173 (0.15252) \\
2 & 0.31077 (0.00102) & 0.08062 (0.00302) & 9.13397 (0.04794) \\
\hline
\end{tabular}
    }
    \label{t1}
\end{table}

In addition, we performed this process with univariate quantile-based measures of dispersion, skewness, and kurtosis, defined in the introduction for each coordinate of the data, for the appropriate distributions. For the kurtosis measures $(\beta,\beta')=(0.2,0.8)$, and otherwise $\beta=0.5$. These depend on the coordinate system and are imperfect as measures for multivariate data, but they provide a quantile-based view, and the results in Table \ref{t2} also validate our choice like Table \ref{t1}. 

\begin{table}
    \centering
    \caption{Estimated coordinate-wise univariate quantile-based measures of dispersion, skewness, and kurtosis for the relevant distributions, where $(\beta,\beta')=(0.2,0.8)$ for the kurtosis measures and $\beta=0.5$ otherwise, with standard errors in parentheses.}
    {\small
\begin{tabular}{ |c||c|c|c|c|c|c| } 
\hline
\multirow{2}{*}{$\nu$} & \multicolumn{2}{c|}{$\delta_0^X(\beta)$} & \multicolumn{2}{c|}{$\lvert\gamma_0^X(\beta)\rvert$} & \multicolumn{2}{c|}{$\kappa_0^X(\beta,\beta')$} \\
\cline{2-7}
 & $v^1$ & $v^2$ & $v^1$ & $v^2$ & $v^1$ & $v^2$ \\
\hline
\multirow{2}{*}{0} & 0.67053 & 1.34786 & 0.14036 & 0.15012 & 5.49972 & 5.59444 \\ 
 & (0.00225) & (0.00472) & (0.00412) & (0.00393) & (0.03086) & (0.03288) \\ 
\multirow{2}{*}{1} & 0.67144 & 0.67401 & 0.06140 & 0.06421 & 5.40792 & 5.44465 \\
 & (0.00236) & (0.00225) & (0.00398) & (0.00381) & (0.03068) & (0.03136) \\ 
\multirow{2}{*}{2} & 0.67038 & 0.33348 & 0.00396 & 0.00342 & 5.26431 & 5.28845 \\
 & (0.00211) & (0.00120) & (0.00379) & (0.00400) & (0.03127) & (0.03092) \\
\hline
\end{tabular}
    }
    \label{t2}
\end{table}

Now, we get to our seven measures, each calculated for the relevant distributions, using (\ref{approximations}) with $\Xi_k$ for $k=24$, as described in Remark \ref{weak}. We again used $(\beta,\beta')=(0.2,0.8)$ for the kurtosis measures and $\beta=0.5$ for the others. Table \ref{t3} lists the results; each of the measures decreases as $\nu$ increases, as expected.

\begin{table}
    \centering
    \caption{Estimates for (\ref{approximations}) using $\Xi_k$, $k=24$, as described in Remark \ref{weak}, for the relevant distributions, where $(\beta,\beta')=(0.2,0.8)$ for the kurtosis measures and $\beta=0.5$ otherwise, with standard errors in parentheses.}
    {\small
\begin{tabular}{ |c||c|c|c|c|c|c|c| } 
\hline
$\nu$ & $\hat{\delta}_1^X(\beta;\Xi_k)$ & $\hat{\delta}_2^X(\beta;\Xi_k)$ & $\hat{\gamma}_1^X(\beta;\Xi_k)$ & $\lVert\hat{\gamma}_2^X(\beta;\Xi_k)\rVert$ & $\hat{\kappa}_1^X(\beta,\beta';\Xi_k)$ & $\hat{\kappa}_2^X(\beta,\beta';\Xi_k)$ & $\hat{\alpha}^X(\beta;\Xi_k)$\\
\hline
\multirow{2}{*}{0} & 1.52211 & 1.32734 & 0.13874 & 0.01161 & 5.65998 & 6.03159 & 0.30551 \\ 
 & (0.00390) & (0.00287) & (0.00208) & (0.00036) & (0.01941) & (0.01616) & (0.00340) \\ 
\multirow{2}{*}{1} & 0.89785 & 0.87429 & 0.07861 & 0.00357 & 5.51208 & 5.87081 & 0.17869 \\ 
 & (0.00182) & (0.00171) & (0.00149) & (0.00016) & (0.01937) & (0.01561) & (0.00278) \\ 
\multirow{2}{*}{2} & 0.75481 & 0.65864 & 0.05628 & 0.00178 & 5.35924 & 5.71843 & 0.13664 \\ 
 & (0.00179) & (0.00136) & (0.00118) & (0.00009) & (0.01894) & (0.01512) & (0.00209) \\ 
\hline
\end{tabular}
    }
    \label{t3}
\end{table}

To alleviate concerns about extreme geometric quantiles, we also performed the same process when $(\beta,\beta')=(0.2,0.98)$ for the kurtosis measures and $\beta=0.98$ otherwise. The results are shown in Table \ref{t5}, with Table \ref{t4} corresponding to Table \ref{t2}. We have similar conclusions.

\begin{table}[h!]
    \centering
    \caption{Estimated coordinate-wise univariate quantile-based measures of dispersion, skewness, and kurtosis for the relevant distributions, where $(\beta,\beta')=(0.2,0.98)$ for the kurtosis measures and $\beta=0.98$ otherwise, with standard errors in parentheses.}
    {\small
\begin{tabular}{ |c||c|c|c|c|c|c| } 
\hline
\multirow{2}{*}{$\nu$} & \multicolumn{2}{c|}{$\delta_0^X(\beta)$} & \multicolumn{2}{c|}{$\lvert\gamma_0^X(\beta)\rvert$} & \multicolumn{2}{c|}{$\kappa_0^X(\beta,\beta')$} \\
\cline{2-7}
 & $v^1$ & $v^2$ & $v^1$ & $v^2$ & $v^1$ & $v^2$ \\
\hline
\multirow{2}{*}{0} & 2.26768 & 4.50882 & 0.40440 & 0.41051 & 12.28863 & 12.40258 \\ 
 & (0.00703) & (0.01424) & (0.00276) & (0.00254) & (0.09263) & (0.09731) \\ 
\multirow{2}{*}{1} & 2.25123 & 2.26780 & 0.17675 & 0.16988 & 11.24632 & 11.21419 \\
 & (0.00741) & (0.00673) & (0.00328) & (0.00325) & (0.07421) & (0.08626) \\ 
\multirow{2}{*}{2} & 2.26861 & 1.13193 & 0.00800 & 0.00275 & 10.04283 & 10.07831 \\
 & (0.00709) & (0.00347) & (0.00323) & (0.00328) & (0.06891) & (0.07318) \\
\hline
\end{tabular}
    }
    \label{t4}
\end{table}

\begin{table}[h!]
    \centering
    \caption{Estimates for (\ref{approximations}) using $\Xi_k$, $k=24$, as described in Remark \ref{weak}, for the relevant distributions, where $(\beta,\beta')=(0.2,0.98)$ for the kurtosis measures and $\beta=0.98$ otherwise, with standard errors in parentheses.}
    {\small
\begin{tabular}{ |c||c|c|c|c|c|c|c| } 
\hline
$\nu$ & $\hat{\delta}_1^X(\beta;\Xi_k)$ & $\hat{\delta}_2^X(\beta;\Xi_k)$ & $\hat{\gamma}_1^X(\beta;\Xi_k)$ & $\lVert\hat{\gamma}_2^X(\beta;\Xi_k)\rVert$ & $\hat{\kappa}_1^X(\beta,\beta';\Xi_k)$ & $\hat{\kappa}_2^X(\beta,\beta';\Xi_k)$ & $\hat{\alpha}^X(\beta;\Xi_k)$\\
\hline
\multirow{2}{*}{0} & 9.68865 & 8.02046 & 0.15085 & 0.01102 & 19.70886 & 19.73841 & 0.32443 \\ 
 & (0.01920) & (0.01371) & (0.00111) & (0.00017) & (0.07213) & (0.06882) & (0.00191) \\ 
\multirow{2}{*}{1} & 5.20339 & 5.05801 & 0.07876 & 0.00299 & 18.83241 & 18.60669 & 0.17032 \\ 
 & (0.00820) & (0.00716) & (0.00124) & (0.00009) & (0.06449) & (0.05936) & (0.00180) \\ 
\multirow{2}{*}{2} & 4.84549 & 4.01105 & 0.02804 & 0.00029 & 17.94682 & 17.44866 & 0.09779 \\ 
 & (0.00989) & (0.00711) & (0.00069) & (0.00002) & (0.05963) & (0.052123) & (0.00145) \\ 
\hline
\end{tabular}
    }
    \label{t5}
\end{table}

As a note on the choice of $k$, it should be large enough to provide a good approximation, and in light of Proposition \ref{approx}, the larger, the better, so in that sense, $k$ is analogous to the number of repetitions in a permutation test, the bootstrap method and, more generally, the Monte Carlo method. 
To justify our choice of $k=24$, we have recalculated Tables \ref{t4} and \ref{t5} in Tables \ref{t6} and \ref{t7}, respectively, with $k=72$ and the changes are negligible.

\begin{table}[h!]
    \centering
    \caption{Estimates for (\ref{approximations}) using $\Xi_k$, $k=72$, as described in Remark \ref{weak}, for the relevant distributions, where $(\beta,\beta')=(0.2,0.8)$ for the kurtosis measures and $\beta=0.5$ otherwise, with standard errors in parentheses.}
    {\small
\begin{tabular}{ |c||c|c|c|c|c|c|c| } 
\hline
$\nu$ & $\hat{\delta}_1^X(\beta;\Xi_k)$ & $\hat{\delta}_2^X(\beta;\Xi_k)$ & $\hat{\gamma}_1^X(\beta;\Xi_k)$ & $\lVert\hat{\gamma}_2^X(\beta;\Xi_k)\rVert$ & $\hat{\kappa}_1^X(\beta,\beta';\Xi_k)$ & $\hat{\kappa}_2^X(\beta,\beta';\Xi_k)$ & $\hat{\alpha}^X(\beta;\Xi_k)$\\
\hline
\multirow{2}{*}{0} & 1.52468 & 1.32735 & 0.13992 & 0.01161 & 5.66199 & 6.03167 & 0.31020 \\ 
 & (0.00390) & (0.00287) & (0.00208) & (0.00036) & (0.01954) & (0.01618) & (0.00343) \\  
\multirow{2}{*}{1} & 0.89921 & 0.87432 & 0.07979 & 0.00358 & 5.51633 & 5.87119 &  0.18316\\ 
 & (0.00183) & (0.00171) & (0.00149) & (0.00016) & (0.01935) & (0.01561) & (0.00279) \\ 
\multirow{2}{*}{2} & 0.75602 & 0.65864 & 0.05782 & 0.00178 & 5.36248 & 5.71882 & 0.14103 \\ 
 & (0.00179) & (0.00135) & (0.00119) & (0.00009) & (0.01895 & (0.01511) & (0.00209) \\ 
\hline
\end{tabular}
    }
    \label{t6}
\end{table}

\begin{table}[h!]
    \centering
    \caption{Estimates for (\ref{approximations}) using $\Xi_k$, $k=72$, as described in Remark \ref{weak}, for the relevant distributions, where $(\beta,\beta')=(0.2,0.98)$ for the kurtosis measures and $\beta=0.98$ otherwise, with standard errors in parentheses.}
    {\small
\begin{tabular}{ |c||c|c|c|c|c|c|c| } 
\hline
$\nu$ & $\hat{\delta}_1^X(\beta;\Xi_k)$ & $\hat{\delta}_2^X(\beta;\Xi_k)$ & $\hat{\gamma}_1^X(\beta;\Xi_k)$ & $\lVert\hat{\gamma}_2^X(\beta;\Xi_k)\rVert$ & $\hat{\kappa}_1^X(\beta,\beta';\Xi_k)$ & $\hat{\kappa}_2^X(\beta,\beta';\Xi_k)$ & $\hat{\alpha}^X(\beta;\Xi_k)$\\
\hline
\multirow{2}{*}{0} & 9.69202 & 8.02062 & 0.15198 & 0.01102 & 19.693314 & 19.73880 & 0.32757 \\ 
 & (0.01917) & (0.01371) & (0.00112) & (0.00017) & (0.07229) & (0.06884) & (0.00192) \\ 
\multirow{2}{*}{1} & 5.20581 & 5.05801 & 0.08033 & 0.00300 & 18.80979 & 18.60679 & 0.17242 \\ 
 & (0.00822) & (0.00716) & (0.00128) & (0.00009) & (0.06436) & (0.05939) & (0.00183) \\ 
\multirow{2}{*}{2} & 4.84741 & 4.01110 & 0.03034 & 0.00029 & 17.92440 & 17.44927 & 0.09955 \\ 
 & (0.00987) & (0.00710) & (0.00071) & (0.00002) & (0.05957) & (0.05211) & (0.00147) \\ 
\hline
\end{tabular}
    }
    \label{t7}
\end{table}

Our measures act reasonably despite the concerns of \cite{Girard2017} about the shape of (extreme) isoquantile contours. This is not surprising. Suppose a dataset is transformed while maintaining its median. In this case, it is sufficient, though not necessary, for the contours to be pulled inward toward the median for our dispersion measures to decrease, and the specific shapes of the isoquantile contours do not matter. This is observed for our dispersion distributions. See the first column of Figure \ref{mvfig:simquantiles}, which displays how estimated isoquantile contours ($k=24$) for the distributions based on samples of size $N=30000$, centered so that the medians coincide for ease of comparing contours for different distributions, change with $\nu$. The second and fourth columns also clearly show reductions in skewness and spherical asymmetry, respectively, as $\nu$ increases. The third column is more challenging to interpret but is included for completeness. Furthermore, Figure \ref{mvfig:simquantiles} provides a visualization of our measures of the dispersion, skewness, kurtosis, and spherical asymmetry of a distribution.



\begin{figure}[h!]
	\centering
		\includegraphics[width=0.23\linewidth]{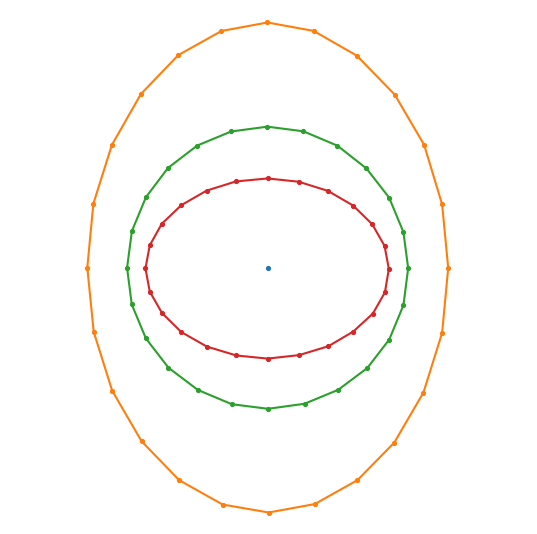}
		\includegraphics[width=0.23\linewidth]{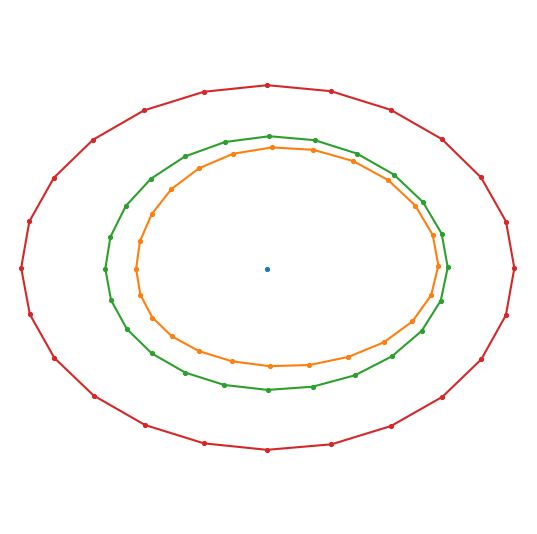}
		\includegraphics[width=0.23\linewidth]{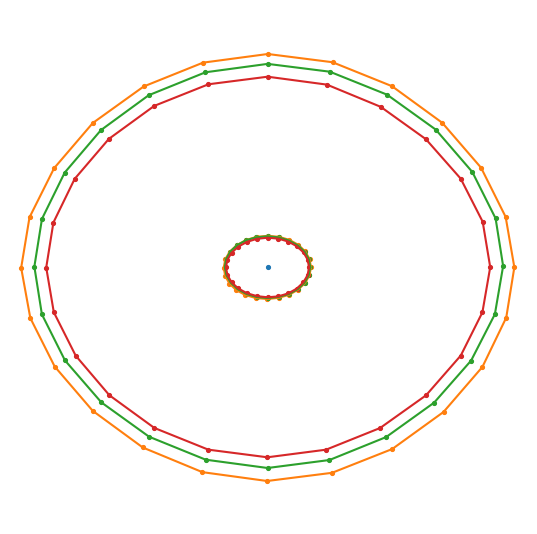}
		\includegraphics[width=0.23\linewidth]{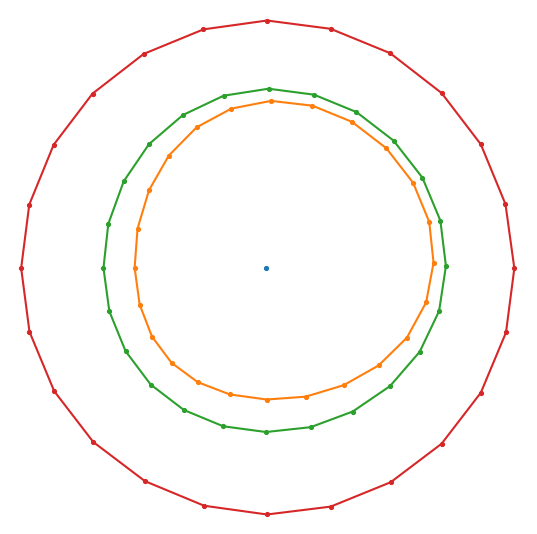}
		\includegraphics[width=0.23\linewidth]{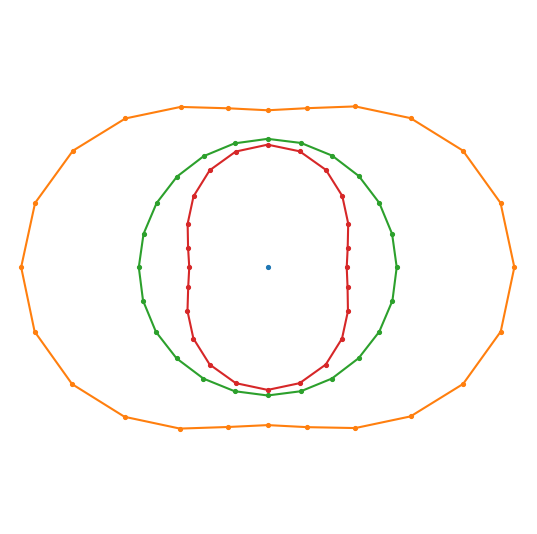}
		\includegraphics[width=0.23\linewidth]{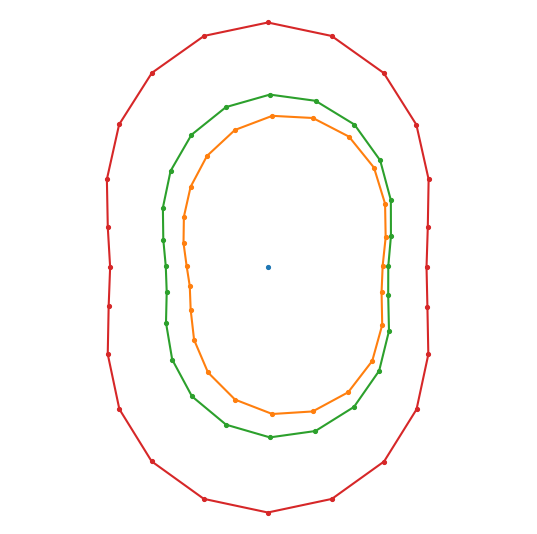}
		\includegraphics[width=0.23\linewidth]{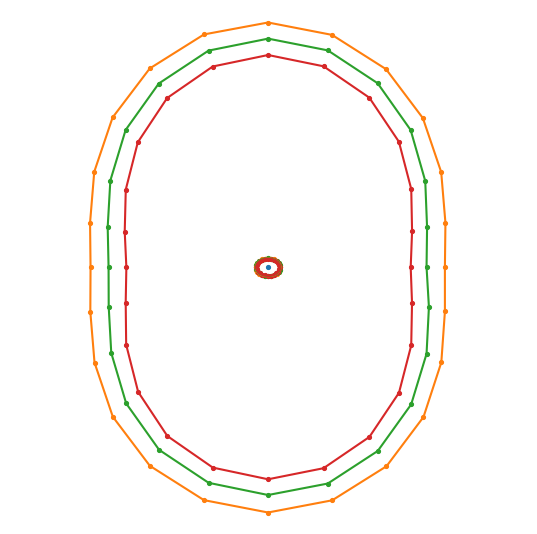}
		\includegraphics[width=0.23\linewidth]{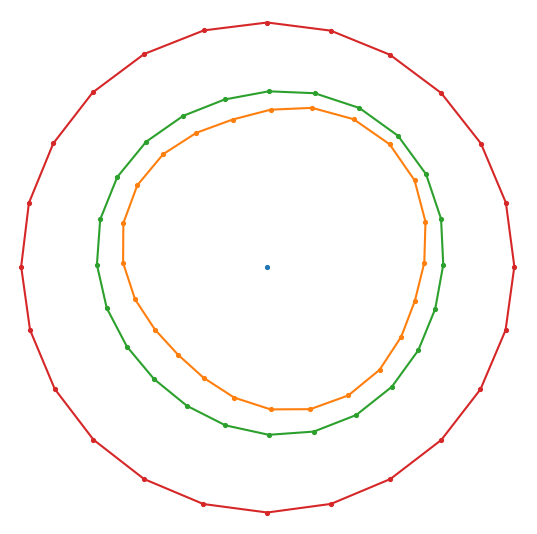}
 \caption{The four columns display estimated isoquantiles contours ($k=24$) for, in order, the dispersion, skewness, kurtosis, and spherical asymmetry distributions, based on samples of size $N=30000$ centered so that the medians are at the origin; plots are not to the same scale. In the first row, $(\beta,\beta')=(0.2,0.8)$ for the kurtosis measures and $\beta=0.5$ otherwise; in the second row, $(\beta,\beta')=(0.2,0.98)$ for the kurtosis measures and $\beta=0.98$ otherwise. The orange, green, and red contours are the estimated contours for the $\nu=0,1,2$ distributions, respectively, while the blue point in the center is their common median.}
	\label{mvfig:simquantiles}
\end{figure}

\subsection{Confidence regions} \label{boot}

We now use bootstrapping to compute confidence regions and, hence, do testing. Recall that for a parameter $\theta$, an estimate $\hat{\theta}$, and bootstrap estimates $\hat{\theta}_1^*,\ldots,\hat{\theta}_T^*$, a possible $100(1-\tau)$\% bootstrap confidence region is $\{\hat{\theta}-w:w\in W\}$, where $W$ is some region containing $100(1-\tau)$\% of the values of $\hat{\theta}_t^*-\hat{\theta}$, $t=1,\ldots,T$. Here, we let $W=B_r^n(0)$ with $r$ chosen for appropriate coverage. This gives a pivotal (also called basic or empirical) bootstrap confidence region. If one were also interested in a percentile bootstrap confidence region, the most obvious candidate would be a ball centered at $\hat{\theta}$ of the form $\{\hat{\theta}+u:u\in B_{r'}^n(0)\}$, with $r'$ chosen so that this set contains $100(1-\tau)\%$ of the values of $\hat{\theta}_t^*$, $t=1,\ldots,T$; this is clearly exactly equivalent to our pivotal bootstrap confidence region, and therefore, has the same coverage.

Setting $T=1000$ and $\alpha=0.05$, for $X$ following the standard bivariate normal distribution, we calculated a confidence region for $\gamma_2^X(\beta)$ (precisely, $\hat{\gamma}_2^X(\beta,\Xi_k)$ for $\Xi_k$, $k=24$, as described in Remark \ref{weak}) based on $N=300$ draws from this distribution; for computational reasons, we have chosen $k=24$ instead of $k=120$, as in the previous section. For this $X$, $\gamma_2^X(\beta)=0$ for all $\beta\in(0,1)$ thanks to Proposition \ref{prop}(a), so we checked whether 0 was contained in the confidence region. To estimate confidence region coverage, this process was repeated 1000 times for $\beta=0.5$ and 0.98. Our confidence regions had excellent coverage, with 96\% and 94.8\% of confidence regions containing 0 when $\beta=0.5$ and 0.98, respectively. This provides some empirical support for our bootstrapped confidence regions, but further investigation is needed into their theoretical properties.

Because $\gamma_2^X(\beta)=0$ for all $\beta\in(0,1)$ by Proposition \ref{prop}(a) for normal $X$, theoretically justified confidence regions could be used to test for the normality of a multivariate distribution by rejecting the null hypothesis of normality when the region for $\gamma_2^X(\beta)=0$ does not contain 0.


\section{Conclusion}

This paper contributes to the growing literature on the uses of geometric quantiles. The measures defined here provide useful summaries of multivariate distributions, even for extreme values of $\beta$, and are easy to calculate.

More helpful research in this area might include computing these measures for different types of distributions, as in Proposition \ref{prop}(a) and (b), and working on their asymptotic properties and testing and confidence regions via bootstrapping or other methods. The adaptation of our measures to functional data is also of interest; as \cite{Chaudhuri1996} notes, multivariate quantiles can naturally be generalized to infinite-dimensional Hilbert space-valued data, and \cite{Chowdhury2019} do so to perform quantile regression for functional data. The measures in this paper also naturally generalize to Hilbert spaces and, thus, to functional data. More work on theoretical and practical considerations when using these measures in functional data analysis, an active field of research as evidenced by \cite{Aneiros2022}, would be welcome.

	\section*{Acknowledgments} 
		
  This research was supported by the National Research Foundation of Korea (NRF) funded by the Korean government (2021R1A2C1091357).

	

	
 \bibliographystyle{apalike}
\bibliography{references}

\begin{thebibliography}{}

\bibitem[Aneiros et~al., 2022]{Aneiros2022}
Aneiros, G., Horov\'a, I., Hu\u{s}kov\'a, M., and Vieu, P. (2022).
\newblock On functional data analysis and related topics.
\newblock {\em Journal of Multivariate Analysis}, 189.

\bibitem[Balanda and MacGillivray, 1988]{Balanda1988}
Balanda, K.~P. and MacGillivray, H.~L. (1988).
\newblock Kurtosis: {A} critical review.
\newblock {\em The American Statistician}, 42:111--119.

\bibitem[Chakraborty, 2001]{Chakraborty2001}
Chakraborty, B. (2001).
\newblock On affine equivariant multivariate quantiles.
\newblock {\em Annals of the Institute of Statistical Mathematics},
  53:380--403.

\bibitem[Chakraborty, 2003]{Chakraborty2003}
Chakraborty, B. (2003).
\newblock On multivariate quantile regression.
\newblock {\em Journal of Statistical Planning and Inference}, 110:109--132.

\bibitem[Chaudhuri, 1996]{Chaudhuri1996}
Chaudhuri, P. (1996).
\newblock On a geometric notion of quantiles for multivariate data.
\newblock {\em Journal of the American Statistical Association}, 91:862--872.

\bibitem[Chowdhury and Chaudhuri, 2019]{Chowdhury2019}
Chowdhury, J. and Chaudhuri, P. (2019).
\newblock Nonparametric depth and quantile regression for functional data.
\newblock {\em Bernoulli}, 25:395--423.

\bibitem[Dudley, 2002]{Dudley2002}
Dudley, R.~M. (2002).
\newblock {\em Real Analysis and Probability}.
\newblock Cambridge University Press, NY.

\bibitem[Girard and Stupfler, 2017]{Girard2017}
Girard, S. and Stupfler, G. (2017).
\newblock Intriguing properties of extreme geometric quantiles.
\newblock {\em REVSTAT--Statistical Journal}, 15:107--139.

\bibitem[Groeneveld and Meeden, 1984]{Groeneveld1984}
Groeneveld, R.~A. and Meeden, G. (1984).
\newblock Measuring skewness and kurtosis.
\newblock {\em Journal of the Royal Statistical Society, Series D},
  33:391--399.

\bibitem[Konen and Paindaveine, 2022]{Konen2022}
Konen, D. and Paindaveine, D. (2022).
\newblock Multivariate $\rho$-quantiles: {A} spatial approach.
\newblock {\em Bernoulli}, 28:1912--1934.

\bibitem[Konen and Paindaveine, 2024]{Konen2024}
Konen, D. and Paindaveine, D. (2024).
\newblock On the robustness of spatial quantiles.
\newblock {\em Annales de l'Institut Henri Poincar\'e', Probabilit\'es et
  Statistiques}.

\bibitem[Mardia, 1970]{Mardia1970}
Mardia, K.~V. (1970).
\newblock Measures of multivariate skewness and kurtosis with applications.
\newblock {\em Biometrika}, 57:519--530.

\bibitem[M\'ori et~al., 1994]{Mori1994}
M\'ori, T.~F., Rohatgi, V.~J., and Sz\'ekely, G.~J. (1994).
\newblock On multivariate skewness and kurtosis.
\newblock {\em Theory of Probability \& its Applications}, 38:547--551.

\bibitem[Varadarajan, 1958]{Varadarajan1958}
Varadarajan, V.~S. (1958).
\newblock On the convergence of sample probability distributions.
\newblock {\em Sankhy\=a: The Indian Journal of Statistics}, 19:23--26.

\end{thebibliography}
	
\end{document}